gsave
newpath
  20 20 moveto
  20 220 lineto
  220 220 lineto
  220 20 lineto
closepath
2 setlinewidth
gsave
  .4 setgray fill
grestore
stroke
grestore
\documentclass{svjour3}                     
\smartqed  
\usepackage{amssymb}
\usepackage{amsmath}
\usepackage{graphicx}
\spnewtheorem{algorithm}{Algorithm}{\bf}{\rm}
\begin{document}

\title{Vine copulas as a mean for the construction of high dimensional probability
distribution associated to a Markov Network
}

\titlerunning{Vine copulas and Markov networks}        

\author{Edith Kov\'{a}cs         \and
        Tam\'{a}s Sz\'{a}ntai 
}

\authorrunning{E. Kov\'{a}cs and T. Sz\'{a}ntai} 

\institute{Department of Mathematics, \'{A}VF College of Management of Budapest \at
              Vill\'{a}nyi \'{u}t 11-13, H-1114 Budapest, Hungary
           \and  
           Institute of Mathematics, Budapest University of Technology and Economics,
            \at M\H{u}egyetem rkp. 3, H-1111 Budapest, Hungary \\
              Tel.: +36-1-463-1298\\
              Fax: +36-1-463-1291\\
              \email{szantai@math.bme.hu}                    
}

\date{Received: date / Accepted: date}

\maketitle

\begin{abstract}
Building higher-dimensional copulas is generally recognized as a difficult problem. Regular-vines using bivariate copulas provide a flexible class of high-dimensional dependency models. In large dimensions, the drawback of the model is the exponentially increasing complexity. Recognizing some of the conditional independences is a possibility for reducing the number of levels of the pair-copula decomposition, and hence to simplify its construction (see \cite{Aasetal09}). The idea of using conditional independences was already performed under elliptical copula assumptions \cite{HanKurCoo06}, \cite{KurCoo02} and in the case of DAGs in a recent work \cite{Bau11}.

	We provide a method which uses some of the conditional independences encoded by the Markov network underlying the variables. We give a theorem which under some graph conditions makes possible to derive pair-copula decomposition of the probability density function associated to a Markov network. 	
	
	As the underlying Markov network is usually unknown, we first have to discover it from the sample data. Using our results published in \cite{SzaKov08} and \cite{KoSza10a} we will show how to derive a multidimensional copula model exploiting the information on conditional independences hidden in the sample data.
\keywords{Copula decomposition \and $t$-cherry junction tree \and Markov network 
\and Cherry-wine probability distribution \and Graphical models}
\end{abstract}

\section{Introduction}
\label{intro}
Copulas in general are known to be useful tool for modeling
multivariate probability distributions since they serve as a link between univariate marginals.
Pair-copula construction introduced by H. Joe \cite{Joe97} is able to encode more
types of dependencies in the same time since they can be expressed as a product of different types of bivariate copulas. For solving the problem wich occurs when
we want to find consistent marginal copulas involved in the expression of
a junction tree copula density (see \cite{KoSza10a}) we found extremly useful the
concept of Regular-vine copulas. Our research in this direction was also motivated by the arising open questions in the papers published in this field, as follows.

The paper \cite{Aasetal09} calls the attention on the fact that ''conditional independence may reduce the number of the
pair-copula decompositions and hence simplify the construction''. 
In this paper the importance of choosing a good
factorisation which takes advantage from the conditional independence
relations between the random variables is pointed out. 
In the present paper we give a method for
findig that pair copula construction which exploits the conditional
independences between the variables of a given Markov network. We also give a
method for constructing Regular-vines starting from a multivariate data set.

The importance of taking into account the conditional independences between the
variables encoded in a Bayesian Network (directed acyclic graph) was
explored in the papers \cite{KurCoo02} and \cite{HanKurCoo06}. Two problems of this aspect are
discussed. First when the Bayesian Network (BN) is known, some of the conditional
independences taken from the BN are used to simplify a given expression of the D- or C- vine
copula. Secondly the problem of reconstruction of the BN from a sample data set was formulated 
under the assumption that the joint distribution is multivariate normal.
For discovering the independences and conditional independences between the
variables in \cite{HanKurCoo06} are used the correlations, the conditional
correlations and the determinant of the correlation matrix. In the present paper we also exploit
the conditional independences encoded in a Markov network which has the advantage that we do not need to know the ordering of the random variables. We will express the conditional independences in terms of information theoretical concepts which do not need any assumption on the type of copula.

In the recent work \cite{Bau11} Bauer et al. are dealing with a more general case with the
pair-copula constructions for non-Gaussian BN. In there paper the BN is
supposed to be known. The formula of probability distribution associated to
the given BN is expressed by pair-copulas. A similar idea will be used in
our approach, we will transform the so called cherry-tree copula introduced in 
\cite{KoSza10a} into a vine copula constructed from pair copula-blocks.

The truncated Regular-vine copula is defined in \cite{Kur10} and \cite{BreCzaAas10}. In \cite{Kur10}  an algorithm is developed for searching the "best vine". This algorithm uses the partial correlations. This paper suggested us the idea to prove a theorem which
ensures the construction of the best truncated Regular-vine distribution, at a given level $k$. In
order to find such a representation we give a greedy algorithm, which
generally is a good heuristic, but if some assumptions are fulfilled the
algorithm results the optimal solution.

Because the work of the present paper is strongly related to Markov networks which also
need some graph theoretical concepts, copulas and the
special case of Regular-vine copulas the second part of the paper is a preliminary part
that contains some of the concepts we will use throughout the paper. The
third part of the paper discusses under which graphical conditions of the Markov network
the multivariate copula can be expressed as a junction tree copula and as a
cherry-tree copula. Then we give a pair-copula construction (formula) and a
Regular-vine structure (graphical structure) of the cherry-tree copula. The
fourth part of the paper presents a method for finding the cherry tree copula starting from a
multivariate sample data set. In the fifth part we discuss the properties of the best fitting probability density and copula density associated to truncated R-vine.
We finish the paper with some conclusions.

\section{Preliminaries}
\label{sec:2}

In this section we introduce some concepts used in graph theory and
probability theory that we need throughout the paper and present how these
can be linked to each other. For a good overview see \cite{LauSpi88}.

\subsection{Markov Network}
\label{subsec:2.1}

We first
present the acyclic hypergraphs and junction trees. We then present a short
reminder on Markov network. We finish this part with the multivariate joint
probability distribution associated to a junction tree.

Let $V=\left\{ 1,\ldots ,d\right\} $ be a set of vertices and $\Gamma $ a
set of subsets of $V$ called {\it set of hyperedges}. A \textit{hypergraph}
consists of a set $V$ of vertices and a set $\Gamma $ of hyperedges. We
denote a hyperedge by $C_i$, where $C_i$ is a subset of $V$. If two vertices
are in the same hyperedge they are connected, which means, the hyperedge of a
hyperhraph is a complete graph on the set of vertices contained in it.

The \textit{acyclic}{\normalsize \ }\textit{hypergraph} is a special type of
hypergraph which fulfills the following requirements:

\begin{itemize}
\item  Neither of the edges of $\Gamma $ is a subset of another edge.

\item  There exists a numbering of edges for which the \textit{running
intersection property} is fullfiled: $\forall j\geq 2\quad \ \exists \ i<j:\
C_i\supset C_j\cap \left( C_1\cup \ldots \cup C_{j-1}\right) $. (Other
formulation is that for all hyperedges $C_i$ and $C_j$ with $i<j-1$,
$C_i \cap C_j \subset C_s \ \mbox{for all} \ s, i<s<j$.)
\end{itemize}

Let $S_j=C_j\cap \left( C_1\cup \ldots \cup C_{j-1}\right) $, for $j>1$ and $%
S_1=\phi $. Let $R_j=C_j\backslash S_j$. We say that $S_j$\textit{separates} 
$R_j$ from $\left( C_1\cup \ldots \cup C_{j-1}\right) \backslash S_j$, and
call $S_j$ separator set or shortly separator.

Now we link these concepts to the terminology of junction trees.

The junction tree is a special tree stucture which is equivalent to the
connected acyclic hypergraphs \cite{LauSpi88}. The nodes of the tree correspond
to the hyperedges of the connected acyclic hypergraph and are called clusters, the edges of the tree
correspond to the separator sets and called separators. The set of all
clusters is denoted by $\mathcal{C}$, the set of all separators is denoted by %
$\mathcal{S}$. The junction tree with the largest cluster containing $k$
variables is called \textit{k-width junction tree}. A vertex which is contained in only one cluster is called {\it simplicial}. The cluster which contains a simplicial is called {\it leaf cluster}.

An important relation between graphs and hypergraphs is given in \cite{LauSpi88}: A
hypergraph is acyclic if and only if it can be considered to be the set of
cliques of a triangulated graph (a graph is triangulated if every cycle of
legth greater than 4 has a chord).

In the Figure \ref{fig:1} one can see a) a triangulated graph, b) the
corresponding acyclic hypergraph and c) the corresponding junction tree.

\begin{figure}
  \includegraphics[bb=80 590 420 770,width=10cm]{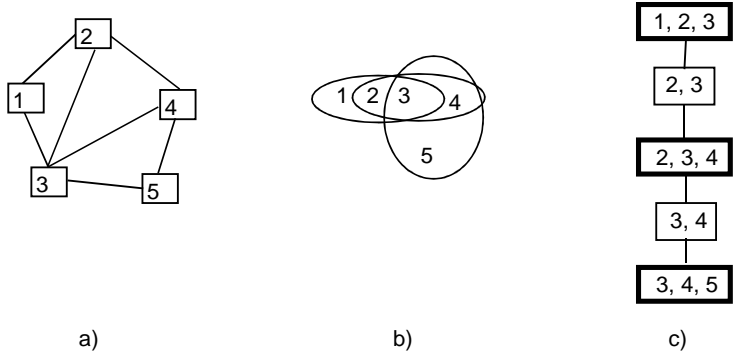}
\caption{a) Triangulated graph, b) The corresponding acyclic hypergraph, c) The corresponding junction tree which is a $t$-cherry junction tree}
\label{fig:1}       
\end{figure}

We consider the random vector $\mathbf{X}=\left( X_1,\ldots ,X_d\right) ^T$, with the
set of indicies $V=\left\{ 1,\ldots ,d\right\} $. Roughly speaking a Markov
network encodes the conditional independences between the random variables.
The graph structure associated to a Markov network consists in the set of
nodes V, and the set of edges $E=\left\{ \left( i,j\right) |i,j\in V\right\}$. 

We say that the probability distribution associated to a Markov network has the global Markov (GM) property \cite{Ham71}if in the graph $\forall A,B,C\subset V$ and C separates A and B in terms of graph then $\mathbf{X}_A$ and $\mathbf{X}_B$ are conditionally independent given $\mathbf{X}_C$, which means in terms of probabilities that 
\[P\left( \mathbf{X}_{A\cup B\cup C}\right) =\dfrac{P\left( \mathbf{X}_{A\cup
C}\right) P\left( \mathbf{X}_{A\cup C}\right) }{P\left( \mathbf{X}_C\right) }.
\]

The concept of \textit{junction tree probability distribution} is related to
the junction tree graph and to the global Markov property of the graph. A
junction tree probability distribution is defined as a product and division of
marginal probability distributions as follows:

\[
P\left( \mathbf{X}\right) =\dfrac{\prod\limits_{C\in \mathcal{C}}P\left(
\mathbf{X}_C\right) }{\prod\limits_{S\in \mathcal{S}}\left[ P\left( \mathbf{X}_S\right)
\right] ^{\nu _S-1}},
\]
where $\mathcal{C}$ is the set of clusters of the
junction tree, $\mathcal S$ is the set of separators, $\nu_S$ is the number of those clusters
which contain the separator S. We emphasize here that the equalities written as
$P(\mathbf{X})=f(P(\mathbf{X}_K), K\in \mathcal{C})$, where $f: \Omega_{\mathbf{X}}\rightarrow R$
hold for any possible realization of $\mathbf{X}$.

\begin{example}
\label{ex:1}
The probability distribution corresponding to Figure \ref{fig:1} is:
\[
\begin{array}{rcl}
P\left( \mathbf{X}\right) & = & \dfrac{P\left( \mathbf{X}_{\{1,2,3\}}\right) P\left(
\mathbf{X}_{\{2,3,4\}}\right) P\left(\mathbf{X}_{\{3,4,5\}}\right) }{P\left(\mathbf{X}_{\{2,3\}}\right)
P\left( \mathbf{X}_{\{3,4\}}\right) }\vspace{3mm} \\
 & = & \dfrac{P\left( X_1,X_2,X_3\right) P\left(
X_2,X_3,X_4\right) P\left( X_3,X_4,X_5\right) }{P\left( X_2,X_3\right)
P\left( X_3,X_4\right) }.
\end{array}
\]
\end{example}

In our paper \cite{SzaKov08} we introduced a special kind of $k$-width junction tree,
called {\it $k$-th order $t$-cherry junction tree} in order to approximate a joint
probability distribution. The $k$-th order $t$-cherry junction tree
probability distribution is assigned to the $k$-th order $t$-cherry tree,
was introduced in \cite{BuPre01}, \cite{BuSza02}.

\begin{definition}
\label{def:def1} 
The recursive construction of the \textit{k-th order
$t$-cherry tree}:

\begin{itemize}
\item  (i) The complete graph of $(k-1)$ nodes from $V$ represent the
smallest $k$-th order \textit{t}-cherry tree;

\item  (ii) By connecting a new vertex $i_k\in V$, with all $\left\{
i_1,\ldots ,i_{k-1}\right\} $ vertices of a $\left( k-1\right) $-
dimensional complete subgraph of the existing $k$-th order $t$-cherry tree, we
obtain a new $k$-th order $t$-cherry tree. $\left\{ \left\{ i_k\right\} \left\{
i_1,\ldots ,i_{k-1}\right\} \right\} $ is called \textit{k}-th order
hypercherry.

\item   (iii) A \textit{k}-th order $t$-cherry tree can be obtained from (i)
by successive application of (ii).
\end{itemize}

The \textit{k}-th order $t$-cherry tree is a special triangulated (chordal or rigid circuit) graph
therefore a junction tree structure is associated to it (see \cite{LauSpi88}).
\end{definition}

\begin{definition}
\label{def:def2}
(\cite{SzaKov08})
The \textit{k-th order $t$-cherry junction tree} 
is defined in the following way:

\begin{itemize}
\item  By using Definition \ref{def:def1} we construct a \textit{k}-th order \textit{t}%
-cherry tree over $V$.

\item  To each hypercherry $\left\{ \left\{ i_k\right\} \left\{ i_1,\ldots
,i_{k-1}\right\} \right\}$ is assigned a cluster $\left\{ i_1,\ldots
,i_{k-1},i_k\right\}$ which represents a node of the junction tree and a separator $\left\{
i_1,\ldots ,i_{k-1}\right\}$ which is an edge of the junction tree.
\end{itemize}

We denote by $\mathcal{C}_{\mbox{ch}}$, and $\mathcal{S}_{\mbox{ch}}$, the
set of clusters and separators of the $t$-cherry junction tree.
\end{definition}

\begin{definition}
\label{def:def3}
(\cite{SzaKov08})
The probability distribution given by (\ref{eq:chd}) and (\ref{eq:chf}) are called {\it t-cherry junction tree probability distribution}

\begin{equation}\label{eq:chd}
P_{\mbox{t-ch}}(\mathbf{X})=\dfrac{\prod\limits_{K\in \mathcal{C}_{\mbox{ch}}}P\left( \mathbf{X}_K\right) }{%
\prod\limits_{S\in \mathcal{S}_{\mbox{ch}}}\left( P\left( \mathbf{X}_S\right) \right) ^{\nu _s-1}}
\end{equation}
in the discrete case and
\begin{equation}\label{eq:chf}
P_{t-ch}\left( \mathbf{X}\right) =\dfrac{\prod\limits_{K\in C_{ch}}f_K\left( 
\mathbf{x}_k\right) }{\prod\limits_{S\in S_{Ch}}\left( f_S\left( \mathbf{x}%
_k\right) \right) ^{\nu _S-1}}
\end{equation}
in the continuous case, where $\nu_S$ denotes the number of clusters which contain the separator $S$.
\end{definition}

\begin{remark}
\label{rem:rem1}
The marginal probability distributions and the density functions involved in the above
formula are marginal probability distributions of $P\left( \mathbf{X}\right) $.
\end{remark}

Example \ref{ex:1} shows a 3-rd order $t$-cherry junction tree probability distribution.

In the following instead of probability distribution associated to a
junction tree we will use shortly junction tree pd and similarly instead of $k$-th order $t$-cherry
tree junction tree distribution we will use shortly $k$-th order $t$-cherry pd. 

\subsection{Copula, Regular-vine copula, junction tree copula and cherry-tree copula}
\label{subsec:2.2}

\begin{definition}\label{def:def4}
A function $C:\left[0;1\right]^{d}\rightarrow\left[0;1\right]$
is called a $d$-dimensional copula if it satisfies the following conditions:

\begin{enumerate}

\item $C\left(u_{1},\ldots,u_{d}\right)$ is increasing in each component
$u_{i}$,

\item $C\left(u_{1},\ldots,u_{i-1},0,u_{i+1},\ldots,u_{d}\right)=0$ for all
$u_{k}\in\left[0;1\right]$,\ $k\neq i,\ i=1,\ldots,n$,

\item $C\left(  1,\ldots,1,u_{i},1,\ldots,1\right)  =u_{i}$ for all $u_{i}%
\in\left[  0;1\right]  ,\ i=1,\ldots,d$,

\item C is $d$-increasing, i.e for all $\left(  u_{1,1},\ldots,u_{1,d}\right)  $
and $\left(  u_{2,1},\ldots,u_{2,d}\right)  $ in $\left[  0;1\right]  ^{d}$
with $u_{1,i}<u_{2,i}$ for all i, we have

\[
\sum\limits_{i_{1}=1}^{2}\cdots
\sum\limits_{i_{d}=1}^{2}\left(  -1\right)  ^{\sum\limits_{j=1}^{d}i_{j}%
}C\left(  u_{i_{1},1},\ldots,u_{i_{d},d}\right)  \geq0.
\]

\end{enumerate}
\end{definition}

Due to Sklar's theorem if $X_{1},\ldots,X_{d}$ are continuous random variables
defined on a common probability space, with the univariate marginal cdf's
$F_{X_{i}}\left(  x_{i}\right)  $ and the joint cdf $F_{X_{1},\ldots,X_{d}%
}\left(  x_{1},\ldots,x_{d}\right)  $ then there exists a unique copula
function $C_{X_{1},\ldots,X_{d}}\left(  u_{1},\ldots,u_{d}\right)  :\left[
0;1\right]  ^{d}\rightarrow\left[  0;1\right]  $ such that by the substitution
$u_{i}=F_{i}\left(  x_{i}\right), \ i=1,\ldots,d$ we get

\[
F_{X_{1},\ldots,X_{d}}\left(x_{1},\ldots,x_{d}\right)
=C_{X_{1},\ldots,X_{d}}\left(F_{1}\left(  x_{1}\right)  ,\ldots,F_{d}\left(  x_{d}\right)  \right)
\]
for all $\left(x_{1},\ldots,x_{d}\right)^{T}\in R^{d}.$

In the following we will use the vectorial notation
$F_{\mathbf{X}_{V}}\left(\mathbf{x}_{V}\right)=C_{X_{V}}\left(
\mathbf{u}_{V}\right)$, where $\mathbf{u}_{V}%
=\left(  F_{X_{i_{1}}}\left(  x_{i_{1}}\right)  ,\ldots,F_{X_{i_{d}}}\left(
x_{i_{d}}\right)  \right)  ^{T}$.

It is known that

\[
f_{X_{i_{1}},\ldots X_{i_{d}}}\left(  x_{i_{1}},\ldots,x_{i_{d}}\right)
=c_{X_{i_{1}},\ldots X_{i_{d}}} \left( F_{X_{i_{1}}}\left(x_{i_{1}}\right),\ldots,F_{X_{i_{d}}}\left(  x_{i_{d}}\right)  \right)
\cdot \prod\limits_{k=1}^{d} f_{X_{i_{k}}} \left(x_{i_{k}}\right)
\]

In vectorial notation this can be written as
\[
f_{\mathbf{X}_{V}}\left(  \mathbf{x}_{V}\right)  =c_{\mathbf{X}_{V}}\left(
\mathbf{u}_{V}\right)  \cdot\prod\limits_{i_{k}\in V}^{{}%
}f_{X_{i_{k}}}\left(  x_{i_{k}}\right)
\]
which can be written as
\[
c_{\mathbf{X}_{V}}\left(  \mathbf{u}_{V}%
\right)  =\dfrac{f_{\mathbf{X}_{V}}\left(  \mathbf{x}_{V}\right)  }%
{\prod\limits_{i_{k}\in V}^{{}}f_{X_{i_{k}}}\left(  x_{i_{k}}\right)  }.
\]

The Regular-vine structures were introduced by T. Bedford and R. Cooke in
\cite{BedCoo01}, \cite{BedCoo02} and described in more detail in \cite{KurCoo06}. 

If it does not cause confusion, instead of $f_{\mathbf{X}_D}$ and $c_{\mathbf{X}_D}$  
we will write $f_D$ and $c_D$. We also introduce the following notations:

\begin{tabular}{lcl}
$F_{i,j|D}$ & -- & the conditional probability distribution function of $X_i$ and $%
X_j$ given $\mathbf{X}_D$;\\
$f_{i,j|D}$ & -- & the conditional probability density function of $X_i$ and $X_j$
given $\mathbf{X}_D$, \\
$c_{i,j|D}$ & -- & the conditional copula density corresponding to  $f_{i,j|D}$,
\end{tabular}
where $D\subset V;i,j\in V\backslash D$.

According to the definition in \cite{KurCoo06}:
\begin{definition}\label{def:def5}
A \textit{Regular-vine (R-vine) on d variables} consists first of a
sequence of trees $T_1,T_2,\ldots ,T_{d-1}$ with nodes $N_i$ and edges $E_i$
for $i=1,\ldots ,d-1$, which satisfy the following conditions.
\end{definition}

\begin{itemize}
\item  $T_1$ has nodes $N_1=\left\{ 1,\ldots ,d\right\} $ and edges $E_1$.

\item  For $i=2,\ldots ,d-1$ the tree $T_i$ has nodes $N_i=E_{i-1}$.

\item  Two edges in tree $T_i$ are joined in tree $T_{i+1}$ only if they
share a common node in tree $T_i$.
\end{itemize}

The last condition usually is referred to as \textit{proximity condition} .

It is shown in \cite{BedCoo01} and \cite{KurCoo06} that the edges in an R-vine
tree can be uniquely identified by two nodes, the conditioned nodes, and a
set of conditioning nodes, i.e., edges are denoted by $e=j\left( e\right)
,k\left( e\right) |D\left( E\right) $ where $D\left( E\right) $ is the
conditioning set. For a good overview see \cite{Cza10}. The next theorem which can be regarded as a central theorem of R-vines see \cite{BedCoo01} links the
probability density function to the copulas assigned to the R- vine
structure. In \cite{BedCoo01} it is shown that there exists a unique probability density assigned to the R-vine, in \cite{BedCoo02} it is shown that this probability distribution can be expressed as (\ref{eq:eq1}).

\begin{theorem}\label{theo:theo1}
The joint density of $\mathbf{X}=\left( X_1,\ldots ,X_d\right) $ is
uniquely determined and given by:
\begin{equation}\label{eq:eq1}
\begin{array}{l}
f\left( x_1,\ldots ,x_d\right)=\left[ \prod\limits_{k=1}^df_k\left(x_k\right) \right]\\
\cdot\prod\limits_{i=1}^{d-1}\prod\limits_{e\in E_i}c_{j\left(
e\right) ,k\left( e\right) |D\left( e\right) }\left( F_{j\left( e\right)
|D\left( e\right) }\left( x_{j\left( e\right) }|\mathbf{x}_{D\left( e\right)
}\right) ,F_{k\left( e\right) |D\left( e\right) }\left( x_{k\left( e\right)
}|\mathbf{x}_{D\left( e\right) }\right) \right).
\end{array}
\end{equation}
\end{theorem}

The arguments of the pair copulas are conditional distribution functions and
can be evaluated using the following expression given by H. Joe \cite{Joe97}
\[
\begin{array}{l}
F_{j\left( e\right) |D\left( e\right) }\left( x_{j\left( e\right) }|\mathbf{x%
}_{D\left( e\right) }\right) \\
=\dfrac{\partial C_{j\left( e\right) ,i|D\left(
e\right) \backslash i}\left( F_{j\left( e\right) |D\left( e\right)
\backslash i}\left( x_{j\left( e\right) }|\mathbf{x}_{D\left( e\right)
\backslash i}\right) ,F_{i|D\left( e\right) \backslash i}\left( x_i|\mathbf{x%
}_{D\left( e\right) \backslash i}\right) \right) }{\partial F_{i|D\left(
e\right) \backslash i}\left( x_i|\mathbf{x}_{D\left( e\right) \backslash
i}\right) }, 
\end{array}
\]
where $i\in D\left( e\right), j\left( e\right) \notin D\left( e\right)$.

We give now an other definition which is related to the $k$-th order $t$-cherry junction tree structure, 
see Definition \ref{def:def2}, which is in fact a $k$-width order uniform hypertree.

\begin{definition}\label{def:def6}
The Regular-vine structure is given by a sequence of 
$t$-cherry junction trees $T_1,T_2,\ldots ,T_{d-1}$ as follows

\begin{itemize}
\item  $T_1$is a regular tree on $V=\left\{ 1,\ldots ,d\right\} $, the set
of edges $E_1=\left\{ e_i^1=\left( l_i,m_i\right)\right. $ $\left.,i=1,\ldots ,d-1,\
l_i,m_i\in V\right\} $; The copula densities $c_{l_i,m_i}\left( F_{l_i}\left(
x_{_{li}}\right) ,F_{m_i}\left( x_{_{m_i}}\right) \right)$ are assigned to
the edges of this tree.

\item  $T_2$ is the second order $t$-cherry junction tree on $%
V=\left\{ 1,\ldots ,d\right\} $, with the set of clusters $E_2=\left\{
e_i^2,i=1,\ldots ,d-1|e_i^2=e_i^1\right\} $ , $\left| e_i^1\right| =2$; the
copula densities

\[
c_{a_{ij}^2,b_{ij}^2|S_{ij}^2}\left( F_{a_{ij}^2|S_{ij}^2}\left(
x_{a_{ij}^2}|\mathbf{x}_{S_{ij}^2}\right) ,F_{b_{ij}^2|S_{ij}^2}\left(
x_{b_{ij}^2}|\mathbf{x}_{S_{ij}^2}\right) \right) 
\]
are assigned to each pair clusters $e_i^2$ and $e_j^2$ , which are linked in
the junction tree $T_2$, where:
\[
\begin{array}{rcl}
S_{ij}^2 &=&e_i^2\cap e_j^2,\vspace{2mm}\\
a_{ij}^2 &=&e_i^2-S_{ij}^2 \vspace{2mm}\\
b_{ij}^2 &=&e_i^2-S_{ij}^2.
\end{array}
\]
\item  $T_k$ is one of the possible $k$-th order $t$-cherry junction tree on $%
V=\left\{ 1,\ldots ,d\right\}$, with the set of clusters $E_k=\left\{
e_i^k,i=1,\ldots ,d-k+1\right\}$ , where each $e_i^k,\left| e_i^k\right| =k$
is obtained from the union of two linked clusters in the $\left( k-1\right) $%
-th order $t$-cherry junction tree $T_{k-1}$ ; The copula densities
\[
c_{a_{ij}^k,b_{ij}^k|S_{ij}^k}\left( F_{a_{ij}^k|S_{ij}^k}\left(
x_{a_{ij}^k}|\mathbf{x}_{S_{ij}^k}\right) ,F_{b_{ij}^k|S_{ij}^k}\left(
x_{b_{ij}^k}|\mathbf{x}_{S_{ij}^k}\right) \right) 
\]
are assigned to each pair of clusters $e_i^k$ and $e_j^k$, which are linked
in the $T_k$ junction tree, where:
\[
\begin{array}{rcl}
S^k &=&e_i^k\cap e_j^k,\vspace{2mm}\\
a_{ij}^k &=&e_i^k-S_{ij}^k \vspace{2mm}\\
b_{ij}^2 &=&e_i^k-S_{ij}^k.
\end{array}
\]
\end{itemize}
\end{definition}

\begin{theorem}\label{theo:theo2}
The Regular-vine probability distribution associated to the R-vine structure given in Definition \ref{def:def6} can be expressed as:
\[
\begin{array}{l}
f\left( x_1,\ldots ,x_d\right) =\left[ \prod\limits_{i=1}^df_i\left(
x_i\right) \right] \left[ \prod\limits_{i=1}^{d-1}c_{e_i^1}\left( F_{l_i}\left(
x_{_{li}}\right) ,F_{l_i}\left( x_{_{li}}\right) \right) \right]\\
\cdot \prod\limits_{i=2}^{d-1}\prod\limits_{e\in
E_i}c_{a_{ij}^k,b_{ij}^k|S_{ij}^k}\left( F_{a_{ij}^k|S_{ij}^k}\left(
x_{a_{ij}^k}|\mathbf{x}_{S_{ij}^k}\right) ,F_{b_{ij}^k|S_{ij}^k}\left(
x_{b_{ij}^k}|\mathbf{x}_{S_{ij}^k}\right) \right).
\end{array}
\]
\end{theorem}

For the following remark see \cite{Aasetal09}, p. 186.

\begin{remark}\label{rem:rem2}
$X_i$ and $X_j$ are conditional independent given the set of variables $%
\mathbf{X}_A, A \subset V\backslash \left\{ i,j\right\} $ if and only if
\end{remark}

\[
c_{ij|A}\left( F_{i|A}\left( x_i|\mathbf{x}_A\right) ,F_{j|A}\left( x_j|%
\mathbf{x}_A\right) \right) =1. 
\]

The following theorem is an important consequence of Theorem \ref{theo:theo1}.

\begin{theorem}\label{theo:theo3}
If in an R-vine the conditional copula densities corresponding to the trees $T_k, T_{k+1}, \ldots, T_{d-1}$ are all equal to 1 then there exists a joint probability distribution which can be expressed only with the
conditional copula densities assigned to $T_1,\ldots ,T_{k-1}$:
\[
\begin{array}{l}
f\left( x_1,\ldots ,x_d\right) =\left[ \prod\limits_{i=1}^df_i\left(
x_i\right) \right] \left[ \prod\limits_{i=1}^{d-1}c_{e_i^1}\left( F_{l_i}\left(
x_{_{li}}\right) ,F_{l_i}\left( x_{_{li}}\right) \right) \right]\\
\cdot \prod\limits_{i=2}^{k-1}\prod\limits_{e\in
E_i}c_{a_{ij}^k,b_{ij}^k|S_{ij}^k}\left( F_{a_{ij}^k|S_{ij}^k}\left(
x_{a_{ij}^k}|\mathbf{x}_{S_{ij}^k}\right) ,F_{b_{ij}^k|S_{ij}^k}\left(
x_{b_{ij}^k}|\mathbf{x}_{S_{ij}^k}\right) \right).
\end{array}
\]
\end{theorem}

The following definition of truncated vine at level $k$ is given in
\cite{BreCzaAas10}.

\begin{definition}\label{def:def7}
A \textit{pairwisely truncated R-vine at level k} (or truncated R-vine at level $k$) is a special R-vine copula with the property that all pair-copulas
with conditioning set equal to, or larger than $k$, are set to bivariate independence copulas.
\end{definition}

There arise the following questions. What special properties have the probability
distribution, if we set to 1 the conditional copula densities associated to
the trees $T_k,\ldots ,T_{d-1}$ of its Regular vine? Which are the
properties of the Markov network associated? We will answer these questions
in Section \ref{sec:3} and Section \ref{sec:5}.

The problem of finding the optimal truncation of the vine structure is
formulated in \cite{Kur10} as follows: '' If we assume that we can assign
the independent copula to nodes of the vine with small absolute values of
partial correlations, then this algorithm can be used to find an optimal
truncation of a vine structure.'' Kurovicka defined as ''best vine'' the one
whose nodes of the top trees (tree with most conditioning) correspond to the
smallest absolute partial correlations. However small partial correlation
result conditional independence only under restrictive assumption, so our
approach deals with a more general case in Section \ref{sec:3}.

In \cite{KoSza10a} we proved a theorem which connects the general junction tree probability distributions with the junction tree copulas. This theorem can be adapted to the $t$-cherry junction trees in the following way.

\begin{theorem}\label{theo:theo4} 
The copula density function associated to a junction tree
probability distribution defined in Definition \ref{def:def3}
\[
f_{\mathbf{X}}\left(  \mathbf{x}\right)
=\dfrac{\prod\limits_{K\in {\mathcal C}_{ch}}f_{\mathbf{X}_{K}}\left(  \mathbf{x}%
_{K}\right)  }{\prod\limits_{S\in\mathcal{S}_{ch}}\left[  f_{\mathbf{X}_{S}}\left(
\mathbf{x}_{S}\right)  \right]  ^{v_{S}-1}},
\]
is given by

\begin{equation}\label{eq:eq2}
c_{\mathbf{X}}\left(  \mathbf{u}_{V}\right)  = \dfrac{\prod\limits_{K\in {\mathcal C}_{ch}
}c_{\mathbf{X}_{K}}\left(  \mathbf{u}_{K}\right)  }{\prod\limits_{S\in
\mathcal{S}_{ch}}\left[  c_{\mathbf{X}_{S}}\left(  \mathbf{u}_{S}\right)  \right]
^{v_{S}-1}}.
\end{equation}
\end{theorem}

\begin{definition}
\label{def:def8}
The copula density given by Formula (\ref{eq:eq2}) is called {\it t-cherry junction tree copula density} or
simply $t$-cherry copula.
\end{definition}

\section{The characteristics of the Markov network associated to a
continuous joint pd which can be expressed as a truncated R-vine}
\label{sec:3}

In this part we refer to Regular-vines as they are defined in Definition \ref{def:def6}. 
First we illustrate the main ideas on an example.

\begin{figure}[h!]
\centering\includegraphics[bb=65 160 515 750,width=8cm]{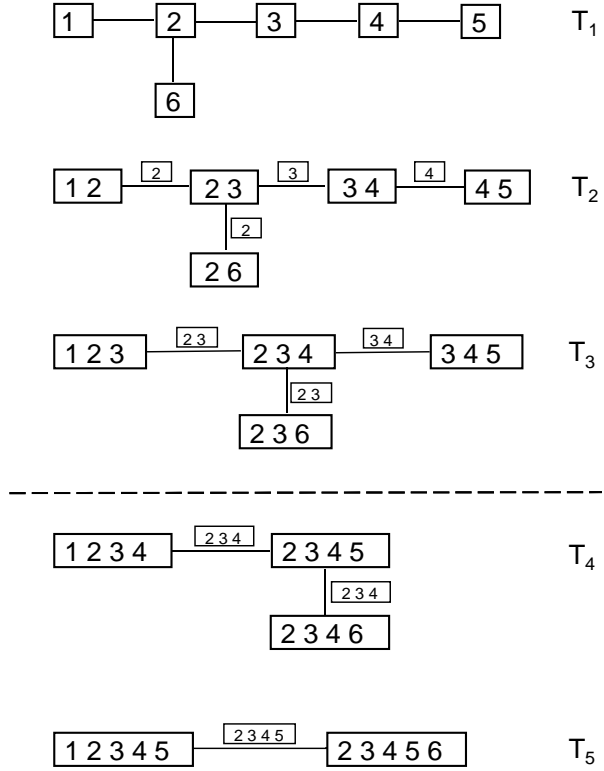}
\caption{Example for an R-vine structure on 6 variables using Definition \ref{def:def6}}
\label{fig:fig2}
\end{figure}

The edge set of the first tree and the sequence of the $t$-cherry trees (in Figure \ref{fig:fig2}) together with the copula densities determined by Definition \ref{def:def6} are the following:

\[
\begin{array}{rl}
T_1: & E_1=\left\{ \left( 1,2\right), \left( 2,3\right), \left(2,6\right), \left( 3,4\right), \left( 4,5\right) \right\}, \\
& c_{1,2},c_{2,3},c_{2,6},c_{3,4},c_{4,5};\\
T_2: & E_2=\left\{e_1^2=\left(1,2\right), e_2^2=\left(2,3\right),e_3^2=\left(2,6\right), e_4^2=\left(3,4\right), 
                  e_5^2=\left(4,5\right)\right\}\\
& S_{1,2}^2=e_1^2\cap e_2^2=\left\{ 2\right\},\\
& \hspace{10mm}a_{1,2}^2=e_1^2-S_{1,2}^2=\left\{ 1\right\},
  b_{1,2}^2=e_2^2-S_{1,2}^2=\left\{ 3\right\},
  c_{a_{1,2}^2,b_{1,2}^2|S_{1,2}^2}=c_{1,3|2}\\
& S_{2,3}^2=e_2^2\cap e_3^2=\left\{ 2\right\},\\
& \hspace{10mm}a_{2,3}^2=e_2^2-S_{2,3}^2=\left\{ 3\right\},
  b_{2,3}^2=e_2^2-S_{2,3}^2=\left\{ 6\right\},
  c_{a_{2,3}^2,b_{2,3}^2|S_{2,3}^2}=c_{3,6|2}\\
& S_{2,4}^2=e_2^2\cap e_4^2=\left\{ 3\right\},\\ 
& \hspace{10mm}a_{2,4}^2=e_2^2-S_{2,4}^2=\left\{ 2\right\},
  b_{2,4}^2=e_4^2-S_{2,4}^2=\left\{ 4\right\},
  c_{a_{2,4}^2,b_{2,4}^2|S_{2,4}^2}=c_{2,4|3}\\
& S_{4,5}^2=e_4^2\cap e_5^2=\left\{ 4\right\},\\
& \hspace{10mm}a_{4,5}^2=e_4^2-S_{4,5}^2=\left\{ 3\right\},
  b_{4,5}^2=e_5^2-S_{4,5}^2=\left\{ 5\right\},
  c_{a_{4,5}^2,b_{4,5}^2|S_{4,5}^2}=c_{3,5|4};\\
T_3: & E_3=\left\{ e_1^3=\left( 1,2,3\right) ,e_2^3=\left( 2,3,4\right)
                  ,e_3^3=\left( 2,3,6\right) ,e_4^3=\left( 3,4,5\right) \right\}\\
& S_{1,2}^3=e_1^3\cap e_2^3=\left\{ 2,3\right\},\\
& \hspace{10mm}a_{1,2}^3=e_1^3-S_{1,2}^3=\left\{ 1\right\},
  b_{1,2}^2=e_2^3-S_{1,2}^3=\left\{ 4\right\},
  c_{a_{1,2}^3,b_{1,2}^3|S_{1,2}^3}=c_{1,4|2,3}\\ 
& S_{2,3}^3=e_2^3\cap e_3^3=\left\{ 2,3\right\},\\
& \hspace{10mm}a_{2,3}^3=e_2^3-S_{2,3}^3=\left\{ 4\right\},
  b_{2,3}^2=e_3^3-S_{2,3}^3=\left\{ 6\right\},
  c_{a_{2,3}^3,b_{2,3}^3|S_{2,3}^3}=c_{4,6|2,3}\\  
& S_{2,4}^3=e_2^3\cap e_4^3=\left\{ 3,4\right\},\\
& \hspace{10mm}a_{2,4}^3=e_2^3-S_{2,4}^3=\left\{ 2\right\},
  b_{2,4}^2=e_4^3-S_{2,4}^3=\left\{ 5\right\},
  c_{a_{2,4}^3,b_{2,4}^3|S_{2,4}^3}=c_{2,5|3,4};\\  
T_4: & E_4=\left\{ e_1^4=\left( 1,2,3,4\right) ,e_2^4=\left(2,3,4,5\right) ,e_3^4=\left( 2,3,4,6\right) \right\}\\
& S_{1,2}^4=e_1^4\cap e_2^4=\left\{ 2,3,4\right\},\\
& \hspace{10mm}a_{1,2}^4=e_1^4-S_{1,2}^4=\left\{ 1\right\},
  b_{1,2}^4=e_2^4-S_{1,2}^4=\left\{ 5\right\},
  c_{a_{1,2}^4,b_{1,2}^4|S_{1,2}^4}=c_{1,5|2,3,4}\\
& S_{2,3}^3=e_2^4\cap e_3^4=\left\{ 2,3,4\right\},\\
& \hspace{10mm}a_{2,3}^4=e_2^4-S_{2,3}^4=\left\{ 5\right\},
  b_{2,3}^4=e_3^4-S_{2,3}^4=\left\{ 6\right\},
  c_{a_{2,3}^4,b_{2,3}^4|S_{2,3}^4}=c_{5,6|2,3,4}\\ 
T_5: & E_5=\left\{ e_1^5=\left( 1,2,3,4,5\right) ,e_2^5=\left(2,3,4,5,6\right) \right\}\\
& S_{1,2}^5=e_1^5\cap e_2^5=\left\{ 2,3,4,5\right\},\\
& \hspace{10mm}a_{1,2}^5=e_1^5-S_{1,2}^5=\left\{ 1\right\},
  b_{1,2}^5=e_2^5-S_{1,2}^5=\left\{ 6\right\},
  c_{a_{1,2}^5,b_{1,2}^5|S_{1,2}^5}=c_{1,6|2,3,4,5}.                    
\end{array}
\]

The joint probability density function of $\mathbf{X=}\left( X_1,\ldots
,X_6\right)$ can be expressed by Theorem \ref{theo:theo2} as follows:
\[
\begin{array}{l}
f\left( x_1,x_2,x_3,x_4,x_5,x_6\right) = \\
= \left( \prod\limits_{i=1}^6f\left(x_i\right) \right) c_{1,2}\left( F_1\left( x_1\right) ,F_2\left( x_2\right)\right) \cdot c_{2,3}\left( F_2\left( x_2\right) ,F_3\left( x_3\right)\right) \cdot c_{2,6}\left( F_2\left( x_2\right) ,F_6\left( x_6\right)\right) \\
\cdot c_{3,4}\left( F_3\left( x_3\right) ,F_4\left( x_4\right) \right)\\
\cdot c_{4,5}\left( F_4\left( x_4\right) ,F_5\left( x_5\right) \right)\\
\cdot c_{1,3|2}\left( F_{1|2}\left( x_1|x_2\right) ,F_{3|2}\left( x_3|x_2\right)\right) \\
\cdot c_{3,6|2}\left( F_{3|2}\left( x_3|x_2\right) ,F_{6|2}\left(x_6|x_2\right) \right) \\
\cdot c_{2,4|3}\left( F_{2|3}\left( x_2|x_3\right) ,F_{4|3}\left( x_4|x_3\right)\right) \\ 
\cdot c_{3,5|4}\left( F_{3|4}\left( x_3|x_4\right) ,F_{5|4}\left(x_5|x_4\right) \right) \\
\cdot c_{1,4|2,3}\left( F_{1|2,3}\left( x_1|x_2,x_3\right) ,F_{4|2,3}\left(x_4|x_2,x_3\right) \right) \\ 
\cdot c_{4,6|2,3}\left( F_{4|2,3}\left(x_4|x_2,x_3\right) ,F_{6|2,3}\left( x_6|x_2,x_3\right) \right) \\
\cdot c_{2,5|3,4}\left( F_{2|3,4}\left( x_2|x_3,x_4\right) ,F_{5|3,4}\left(x_5|x_3,x_4\right) \right) \\
\cdot c_{1,5|2,3,4}\left( F_{1|2,3,4}\left( x_1|x_2,x_3,x_4\right),F_{5|2,3,4}\left( x_5|x_2,x_3,x_4\right) \right) \\
\cdot c_{5,6|2,3,4\text{ }}\left( F_{5|2,3,4}\left( x_1|x_2,x_3,x_4\right),F_{6|2,3,4}\left( x_6|x_2,x_3,x_4\right) \right) \\
\cdot c_{1,6|2,3,4,5}\left( F_{1|2,3,4,5}\left( x_1|x_2,x_3,x_4,x_5\right)
,F_{6|2,3,4,5}\left( x_6|x_2,x_3,x_4,x_5\right) \right)
\end{array}
\]

In this part we regard the graph of the Markov network to be known.
So let us suppose that the Markov network, which encodes the conditional
probabilities between the random variables $X_1,\ldots ,X_6$ is given in Figure \ref{fig:fig3}.

\begin{figure}[h!]
\centering\includegraphics[bb=150 640 510 760,width=6cm]{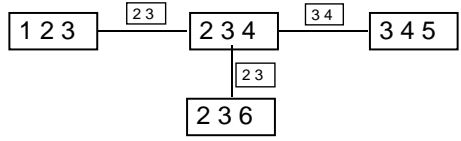}
\caption{$3$-rd order $t$-cherry junction tree}
\label{fig:fig3}
\end{figure}

If the Markov network has the structure in Figure \ref{fig:fig3}, then it is easy to
identify the following conditional independences which are consequences of the Global Markov property:
\[
\begin{array}{lll}
X_1\perp X_4|X_2,X_3; & X_4\perp X_6|X_2,X_3; & X_2\perp X_5|X_3,X_4; \\
X_1\perp X_5|X_2,X_3,X_4; & X_5\perp X_6|X_2,X_3,X_4; & \\
X_1\perp X_6|X_2,X_3,X_4,X_5. &&
\end{array}
\]

Based on the existence of these conditional independences the conditional copula densities associated to the trees $T_3,T_4,T_5$ 
\[
\begin{array}{l}
c_{1,4|2,3}\left( F_{1|23}\left( x_1|x_2,x_3\right) ,F_{4|2,3}\left(x_4|x_2,x_3\right) \right), \\
c_{4,6|2,3}\left( F_{1|23}\left(x_1|x_2,x_3\right) ,F_{4|2,3}\left( x_4|x_2,x_3\right) \right), \\
c_{2,5|3,4}\left( F_{2,5|3,4}\left( x_2|x_3,x_4\right) ,F_{5|3,4}\left(x_5|x_3,x_4\right) \right), \\ c_{1,5|2,3,4}\left( F_{1,5|2,3,4}\left(x_1|x_2,x_3,x_4\right) ,F_{5|2,3,4}\left( x_5|x_2,x_3,x_4\right) \right), \\
c_{5,6|2,3,4\text{ }}\left( F_{5|2,3,4}\left( x_1|x_2,x_3,x_4\right),F_{6|2,3,4}\left( x_6|x_2,x_3,x_4\right) \right), \\
c_{1,6|2,3,4,5}\left( F_{1|2,3,4,5}\left( x_1|x_2,x_3,x_4,x_5\right),
F_{6|2,3,4,5}\left(x_6|x_2,x_3,x_4,x_5\right) \right).
\end{array}
\]
are all equal to 1. We can observe here that a Markov network of the form of a 3-rd order $t$-cherry tree (see Definition \ref{def:def1}) can be expressed as an R-vine truncated at level 3.

This example suggests, that there are $t$-cherry tree probability distributions
which can be represented as a truncated vines.

In the following we suppose the case when the set of separators of the $k$-th order $t$-cherry 
junction tree form a $(k-1)$-th order $t$-cherry junction tree. 
In this case we give an algorithm, which constructs a
Regular-vine structure associated to a $k$-th order $t$-cherry tree
probability distribution (see Definition \ref{def:def3}).

\begin{algorithm}\label{alg:alg1}
Algorithm for obtaining from a $t$-cherry junction tree a truncated Regular-vine construction.

{\it Input}: A $t$-cherry tree structure, ie a set of clusters of size \textit{k},
and the junction tree structure given by the separators.

{\it Output}: A Regular-vine truncated at level $k$.

We obtain recursively an $\left( m-1\right) $ width $t$-cherry junction tree
from a $m$- width $t$-cherry junction tree, for $m=k, \ldots, 1$ as follows:

\begin{itemize}
\item  1. Step. The separators of the $m$-width $t$-cherry tree will be the
clusters in the $(m-1)$-width $t$-cherry tree, which will be linked if between them
is one cluster in the $m$-width $t$-cherry tree, and they are different.

\item  2. Step. The leaf clusters, those clusters which contain a simplicial
node, are transformed into $(m-1)$-width clusters, by deleting a node which is
not simplicial. The cluster obtained in this way will be connected to one of
the clusters obtained in Step 1, which was the separator linked to it in the 
$m$-width $t$-cherry tree junction tree.
\end{itemize}
\end{algorithm}

\begin{definition}\label{def:def9}
The Regular-vine structure obtained from a $t$-cherry tree structure using
Algorithm \ref{alg:alg1} is called \textit{cherry-wine structure}.
\end{definition}

\begin{definition}\label{def:def10}
The joint probability density assigned to a cherry-wine structure is
called \textit{cherry-wine density}, the corresponding copula density is called 
{\it cherry-wine copula density}. 
\end{definition}

\begin{theorem}\label{theo:theo9}
A $t$-cherry copula can be expressed as a cherry-wine copula in the following way:
\[
\begin{array}{ll}
\dfrac{\prod\limits_{K\in {\mathcal C}_{ch}
}c_{K}\left(  \mathbf{u}_{K}\right)  }{\prod\limits_{S\in
\mathcal{S}_{ch}}\left[  c_{S}\left(  \mathbf{u}_{S}\right)  \right]
^{v_{S}-1}}&
=\left[ \prod\limits_{i=1}^{d-1}c_{e_i^1}\left( F_{l_i}\left(
x_{_{li}}\right) ,F_{l_i}\left( x_{_{li}}\right) \right) \right]\\
&\cdot \prod\limits_{i=2}^{k-1}\prod\limits_{e\in
E_i}c_{a_{ij}^k,b_{ij}^k|S_{ij}^k}\left( F_{a_{ij}^k|S_{ij}^k}\left(
x_{a_{ij}^k}|\mathbf{x}_{S_{ij}^k}\right) ,F_{b_{ij}^k|S_{ij}^k}\left(
x_{b_{ij}^k}|\mathbf{x}_{S_{ij}^k}\right) \right).
\end{array}
\]
\end{theorem}

\begin{figure}[h!]
\centering\includegraphics[bb=50 70 520 760,width=8cm]{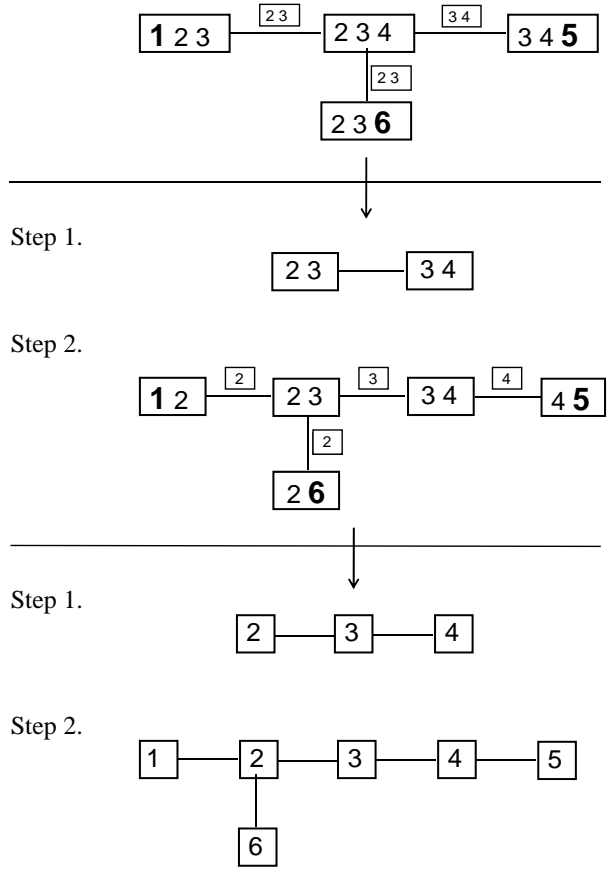}
\caption{Application of Algorithm 1 to a 3-rd order $t$-cherry junction tree in order to obtain a 3-rd order cherry-wine structure}
\label{fig:fig4}
\end{figure}

Example in Figure \ref{fig:fig4} shows how to apply Algorithm \ref{alg:alg1} to a given 3-rd
order $t$-cherry junction tree to obtain a cherry-wine structure.

The cherry wine probability distribution assigned to the 3-rd order cherry-wine 
structure in Figure \ref{fig:fig4} is:

\[
\begin{array}{l}
f\left( x_1,x_2,x_3,x_4,x_5,x_6\right) = \\
= \left( \prod\limits_{i=1}^6f\left(x_i\right) \right) \cdot c_{1,2}\left( F_1\left( x_1\right) ,F_2\left( x_2\right)\right) \cdot c_{2,3}\left( F_2\left( x_2\right) ,F_3\left( x_3\right)\right)  \\
\cdot c_{2,6}\left( F_2\left( x_2\right) ,F_6\left( x_6\right)\right) \cdot c_{3,4}\left( F_3\left( x_3\right) ,F_4\left( x_4\right) \right)
\cdot c_{4,5}\left( F_4\left( x_4\right) ,F_5\left( x_5\right) \right)\\
\cdot c_{1,3|2}\left( F_{1|2}\left( x_1|x_2\right) ,F_{3|2}\left( x_3|x_2\right)\right) 
\cdot c_{3,6|2}\left( F_{3|2}\left( x_3|x_2\right) ,F_{6|2}\left(x_6|x_2\right) \right) \\
\cdot c_{2,4|3}\left( F_{2|3}\left( x_2|x_3\right) ,F_{4|3}\left( x_4|x_3\right)\right) 
\cdot c_{3,5|4}\left( F_{3|4}\left( x_3|x_4\right) ,F_{5|4}\left(x_5|x_4\right) \right) \\
\end{array}
\]

\begin{remark}\label{rem:rem3}
Applying Algorithm \ref{alg:alg1} can result more cherry-wine structures since in Step 2 we can proceed in different directions.
\end{remark}

Starting from the 3-rd order $t$-cherry junction tree given in Figure \ref{fig:fig3} we can obtain $2^{\#\text{leaf clusters}}=8$ 2-nd order $t$-cherry
trees. In the last step there is only one possibility to construct the first tree. We
emphasize here that, if the Markov network has the 3-rd order $t$-cherry tree
structure in Figure \ref{fig:fig3}, than from the 23,040 possible R-vines (see
\cite{Dis10}, \cite{MorCooKur10}) remain only 8.

The question, which arises here is whether the 3-rd order $t$-cherry junction
tree is not a very special structure.

We proved in \cite{SzaKov08} the following theorem by a constructive method:

\begin{theorem}\label{theo:theo5}
Any $k$-width junction tree probability distribution can be expressed as a $k$-th order $t$-cherry tree probability distribution.
\end{theorem}

\begin{remark}\label{rem:rem4} 
There exists more expressions for the $t$-cherry probability distribution, but much smaller
number than Regular-vines, and has the advantage of exploiting the conditional
independences.
\end{remark}

\section{A model selection for special R-vines called cherry-wines}
\label{sec:4}

Now we suppose to have a sample data set. Starting from this dataset we want
to find a good fitting probability distribution. The main idea is fitting
the copula function and the marginal probability distribution separately.
Using pair-vine constructions we will express the joint density function only by
marginal distributions and bivariate (pair)-copulas. First we will search
for a good fitting regular vine structure. As it is shown in \cite{MorCooKur10} the number
of possible regular vines grows exponentially with the number of variables. So
the basic idea is searching through truncated R-vine copulas at a given
level $k$.

Full inference for pair-copula decomposition should in principle consider
three elements \cite{Aasetal09}:

\begin{itemize}
\item  The selection of a specific factorization;
\item  The choice of pair-copula types;
\item  The estimation of parameters of the chosen pair-copulas.
\end{itemize}

This paper is concerned with finding of factorization which exploits some of the conditional independences between the random variables.

There are many papers dealing with selecting specific
Regular-vines as C-vine or D-vine see for example \cite{Aasetal09}.

The main idea of our approach is finding a $t$-cherry copula and then transforming 
it by Algorithm 1 into a cherry-wine copula, which depends just
on pair-copulas. So we will start at a given level $k$, search for the best
fitting $t$-cherry copula to the sample data and find then the factorization
which results the chosen $k$-th order $t$-cherry tree.

\subsection{The Sample derivated copula }

The empirical probability distribution of the sample data is a discrete
multivariate probability distribution. If this data is drawn i.i.d from a
continuous joint probability distribution all realizations are different
vectors. So the joint probability distribution is uniform. The range of each
random variable is equal to the sample size $N$.

As it is shown in \cite{KoSza10a} we first make a partition of the range of
each random variable involved. The intervals obtained contain the same
number of data. We introduced a special type of copula called {\it sample
derivated copula}.

We denote the set of the values of $X_{i}$ in the sample by $\Lambda _{i}$. This set
contains $N$ values, for each random variable. The theoretical range of the
continuous random variable $X_{i}$ will be denoted by $\overline{\Lambda}_{i}$. For
every $i$ we denote by $\lambda _{i}^{m}=\min \overline{\Lambda}_{i}\in R$
and by $\lambda _{i}^{M}=\max \overline{\Lambda}_{i}\in R$. We suppose for
simplicity that $\min \overline{\Lambda}_{i}\neq $ $\min \Lambda _{i}$ and $%
\max \overline{\Lambda}_{i}\neq \max \Lambda _{i}$ For each random variable
$X_{i}$ we define a partition of $\Lambda _{i}$ by $\mathcal{P}_{i}=\left\{ x_{0}^{p_{i}}=\lambda
_{i}^{m},x_{1}^{p_{i}},\ldots ,x_{m_{i}-1}^{p_{i}},x_{m_{i}}^{p_{i}}=\lambda
_{i}^{M}\right\} $ with the following properties:

\begin{itemize}
\item  For each random variable $X_{i}$, each interval $\left(
x_{j-1}^{p_{i}};x_{j}^{p_{i}}\right] ,j=1,\ldots, m_{i}$ contains a given $%
n_{i}=\dfrac{N}{m_{i}}\in N$ number of values from the set $\Lambda _{i}$.

\item  Each $x_{j}^{p_{i}}\in \Lambda _{i},\ j=1,\ldots, m_{i}-1$.
\end{itemize}

The  partition with the above properties will be called uniform partition.
We denote by ${\mathcal P}$ the set of partitions $\{{\mathcal P}_{1},\ldots, {\mathcal P}_{n}\}$.

Let be $\widetilde{X_{i}}$ the categorical random variable associated to the
random variable $X_{i}$:
\[
P(\widetilde{X_{i}} \in \left( x_{j-1}^{p_{i}};x_{j}^{p_{i}}\right])=\dfrac{1}{m_{i}}, j=1,\ldots , m_{i}.
\]
We assign to each $x^{i}\in \left( x_{j-1}^{p_{i}};x_{j}^{p_{i}}\right] $
the number $u_{j}^{i}=\dfrac{j}{m_{i}}, j=0,\ldots, m_{i}$. Obviously
$u_{0}^{i}=0$ and $u_{m_{i}}^{i}=1$. Let 
$\widetilde{\Lambda }_{i}=\left\{ u_{j}^{i}|j=0,\ldots , m_{i}\right\} $.
So we can define the following discrete uniform random variables:
\[
\widetilde{U}_{i}=
\left(
\begin{array}{ccccc}
u_{0}^{i} & u_{1}^{i}      & \ldots & u_{m_{i}-1}^{i} & u_{m_{i}}^{i}   \vspace{3mm}\\
0         &\dfrac{1}{m_{i}} & \ldots & \dfrac{1}{m_{i}} & \dfrac{1}{m_{i}}
\end{array}
\right),
i=1,\ldots, d.
\] 
Now we transform the sample using the above assignment. 
We denote the transformed sample by ${\mathcal T}$. 

\begin{definition}\label{def:def11} 
The function 
$\widetilde{c}:\prod\limits_{i=1}^{d}\widetilde{\Lambda}_{i}\rightarrow R$ 
defined by
\[
\left( u_{k_{1}}^{1},\ldots ,u_{k_{d}}^{d}\right) \mapsto 
\widetilde{c}\left(
u_{k_{1}}^{1},\ldots ,u_{k_{d}}^{d}\right) =\dfrac{\#\left\{(u_{k_{1}}^{1},\ldots ,u_{k_{d}}^{d})\in {\mathcal T}\right\}}{N},
k_{i}=0,\ldots, m_{i}  
\]
will be called {\it sample derivated copula density}.
\end{definition}

In Remark 6 of the paper \cite{KoSza10a} we proved also, that partitioning in this way
the information content of the joint probability distribution depends just
on the sample derivated copula.

The sample derivated copula can be treated as a discrete multivariate
probability distribution. One of its advantages is that the range of the
variables involved are significantly decreased.

Now using the greedy Sz\'{a}ntai-Kov\'{a}cs algorithm introduced in
\cite{SzaKov10a} we find the $k$-th order $t$-cherry copula. The goodness of
fit to the data is quantified by Kullback-Leibler divergence. We emphasize here that
finding the best fitting t-cherry copula is an NP-hard problem 
for $k>2$, but there are cases, when the greedy algorithm finds the
optimal solution, see \cite{SzaKov10a}.

\subsection{The Sz\'{a}ntai-Kov\'{a}cs greedy algorithm}

We present here the algorithm introduced in \cite{SzaKov10a}.

The following theorem regarded to discrete probability distributions given
in \cite{KoSza10a}.

In \cite{SzaKov08} the authors give the following theorem.

\begin{theorem}
\label{theo:2}
The Kullback-Leibler divergence between the true $P(\mathbf{X})$ and the
approximation given by the $k$-width junction tree probability distribution $P(\mathbf{X}_{J})$, determined by the set
of clusters $\mathcal{C}$ and the set of separators $\mathcal{S}$ is :
\begin{equation}
\label{equ:equ1}
\begin{array}{rcl}
KL\left( P\left( \mathbf{X}\right) ,P_{J}\left( \mathbf{X}\right) \right)
& = & -H\left( \mathbf{X}\right) -\left( \sum\limits_{C\in \mathcal{C}}I\left(
\mathbf{X}_C\right) -\sum\limits_{S\in \mathcal{S}}\left( \nu _S-1\right) I\left(
\mathbf{X}_S\right) \right) \\
& + & \sum\limits_{i=1}^dH\left( X_i\right),
\end{array}
\end{equation}
where $I(\mathbf{X}_{C})=\sum\limits_{i\in \mathcal{C}} H\left(X_i\right) - H\left(
\mathbf{X}_{C}\right)$ represents the information content of the random vector $\mathbf{X}_{C}$
and similarly $I(\mathbf{X}_{S})=\sum\limits_{i\in \mathcal{S}} H\left(X_i\right) - H\left(
\mathbf{X}_{S}\right)$ represents the information content of the random vector $\mathbf{X}_{S}$.
\end{theorem}

In Formula (\ref{equ:equ1}) $-H\left( \mathbf{X}\right) +\sum\limits_{i=1}^dH\left(
X_i\right) =I\left( \mathbf{X}\right) $ is independent from the structure of
the junction tree. It is easy to see that minimizing the Kullback-Leibler
divergence means maximizing $\sum\limits_{C\in \mathcal{C}}I\left(
\mathbf{X}_C\right) -\sum\limits_{S\in \mathcal{S}}\left( \nu _S-1\right) I\left(
\mathbf{X}_S\right) $. We call this sum as \textit{weight of the junction tree pd}. As larger
this weight is, as better fits the approximation associated to the junction
tree pd to the true probability distribution. It is well known that $KL=0$ if $P(\mathbf{X})=P_{J}\left( \mathbf{X}\right)$.

\begin{definition}
\label{def:def4} 
We define the following concepts:

\begin{itemize}
\item  the search space:

\vspace{2mm}
$E=\left\{ \chi _{i_k\left( i_1,\ldots ,i_{k-1}\right) }=\left\{ \left\{
X_{i_k}\right\} ,\left\{ X_{i_1},\ldots ,X_{i_{k-1}}\right\} \right\}
|X_{i_1},\ldots ,X_{i_{k-1}},X_{i_k}\in X\right\} $,
\vspace{2mm}

\item  the independence set:

\vspace{2mm}
$\mathcal{F}=\phi \cup \left\{ t-\mbox{cherry junction tree structure}\right\} $,
\vspace{2mm}

\item  the weight function:

\vspace{2mm}
$w:E\rightarrow R\qquad w\left( \chi _{i_k\left( i_1,\ldots ,i_{k-1}\right)
}\right) =I\left( X_{i_1},\ldots ,X_{i_{k-1}},X_{i_k}\right) -I\left(
X_{i_1},\ldots ,X_{i_{k-1}}\right) $.
\end{itemize}
\end{definition}

\begin{algorithm}
\label{alg:alg2}
Sz\'{a}ntai-Kov\'{a}cs's greedy algorithm.

\textit{Input}: Elements of E and their weights which can be calculated
based on the $k$-th order marginal probability distributions.

\textit{Output}: set A which contains the clusters of the $k$-th order $t$%
-cherry juntion tree pd and the wheight of the $k$-th order $t$-cherry junction
tree pd.

\textit{The algorithm}:

$A:=\phi $

Sort $E$ into monotonically decreasing order by wheight $w$;

Choose $x={\arg \max }_{x\in E}\left( w\left( x\right) \right)$;

\qquad let $A:=A\cup \left\{ x\right\} ;\quad E:=E\backslash \left\{
x\right\} ;\quad w:=I\left( x\right)$;

Do for each $x\in E$ taken in monotonically decreasing order

\qquad if $A\cup \left\{ x\right\} \in \mathcal{F}$ then let $A:=A\cup \left\{ x\right\}
;\quad E:=E\backslash \left\{ x\right\} ;\quad w:=w+w\left( x\right) ;$

\qquad if the union of subsets of $A$ is $X$, then Stop;

\qquad else take the next element of $E$.
\end{algorithm}

\subsection{Building the cherry-wine associated to the $t$-cherry tree.}

We calculate the $k$-th order marginal pd from the sample derivated copula.
Using their information content we can define the weights of the elements of
the search space E.

Applying Sz\'{a}ntai-Kov\'{a}cs'a algorithm we obtain a good fitting $t$-cherry tree copula. 

We assign to this the $k$-th order $t$-cherry tree the $T_k$ tree of a regular
vine. Applying now Algorithm \ref{alg:alg1} we can find the corresponding cherry-wine
structure, and using this the expression of the cherry-wine copula density
expressed by pair-copulas.

Now comes the next step the choice of pair-copula types and the estimation
of parameters. For choosing pair copulas we have a large amount of
copula-families, with different properties, tail-dependencies see in \cite{Joe97}, 
\cite{GenRiv93} and \cite{Ne99}. 

\section{Properties of the best fitting cherry-wine probability density,
and cherry-wine copula density}
\label{sec:5}

In this section we discuss the properties of
the best fitting cherry-wine probability density and corresponding copula
density, which are associated to an R-vine truncated at level $k$ from a theoretical point of view. 

We will use the following notations:

\begin{itemize}
\item   $f_{V}\left( \mathbf{x}_{V}\right) $ denotes the joint probability
density of $\mathbf{X}_{V}$, $f_K\left( \mathbf{x}_K\right) $ is the marginal
density of $f_{V}\left( \mathbf{x}_{V}\right) $, where $K\subset V$.

\item   $c_{V}\left( \mathbf{u}_{V}\right) $ denotes the joint copula density
associated to the joint probability density $f_{V}\left( \mathbf{x}_{V}\right) $%
, $c_K\left( \mathbf{u}_K\right) $ is its marginal density which is the
copula density corresponding to $f_K\left( \mathbf{x}_K\right) $

\item  $\widehat{f}_{V_{\mathcal{CS}}}$ denotes the joint $k$-th order
cherry-wine density, associated to a $k$-th order $t$-cherry junction tree
with cluster $\mathcal{C}$ and separator set $\mathcal{S}$ , given by:
\end{itemize}

\begin{equation}\label{eq:eq3}
\widehat{f}_{V_{\mathcal{CS}}}\left( \mathbf{x}_{V}\right) =\dfrac{\prod\limits_{K\in \mathcal{C}%
_{ch}}f_K\left( \mathbf{x}_K\right) }{\prod\limits_{S\in \mathcal{S}%
_{ch}}\left( f_S\left( \mathbf{x}_S\right) \right) ^{\nu _S-1}},
\end{equation}
where $\nu_S$ is the number of clusters which contain $S$.

\begin{theorem}\label{theo:theo8}
The Kullback-Leibler divergence between $f_{V}\left( \mathbf{x}_{V}\right) $ and
the approximating probability density assigned to the cherry-wine $\widehat{%
f}_{V_{\mathcal{CS}}}$, is given by the formula:

\begin{equation}\label{eq:eq4}
KL\left( \widehat{f}_{V_{\mathcal{CS}}}\left( \mathbf{x}_{V}\right) ,f_{V}\left( 
\mathbf{x}_{V}\right) \right) =I\left( \mathbf{X}_{V}\right) -\left[
\sum\limits_{K\in \mathcal{C}_{ch}}I\left( \mathbf{X}_K\right) -\sum_{S\in 
\mathcal{S}_{ch}}\left( \nu _S-1\right) I\left( \mathbf{X}_S\right) \right].
\end{equation}
\end{theorem}

\begin{proof}
\[
\begin{array}{l}
KL\left( \widehat{f}_{V_{\mathcal{CS}}}\left( \mathbf{x}_{V}\right) ,f_{V}\left( \mathbf{x}_{V}\right) \right)
=\int\limits_{R^d}f_{V}\left( \mathbf{x}\right) \log_2\dfrac{f_{V}\left( \mathbf{x}\right) }{\widehat{f}_{V_{\mathcal{CS}}}\left( 
\mathbf{x}\right)} d\mathbf{x}\vspace{3mm}\\
=\int\limits_{R^d}f_{V}\left( \mathbf{x}\right)
\log_2f_{V}\left( \mathbf{x}\right) d\mathbf{x}-\int\limits_{R^d}f_{V}\left( 
\mathbf{x}\right) \log_2\widehat{f}_{V_{\mathcal{CS}}}\left( \mathbf{x}%
\right) d\mathbf{x}\\
=-H\left( \mathbf{X}\right) -\int\limits_{R^d}f_{V}\left( \mathbf{x}\right)
\log_2\dfrac{\prod\limits_{K\in \mathcal{C}}f_K\left( \mathbf{x}_K\right) }{%
\prod\limits_{S\in \mathcal{S}}\left( f_S\left( \mathbf{x}_S\right) \right)
^{\nu _S-1}}d\mathbf{x}\\
=-H\left( \mathbf{X}\right)-\int\limits_{R^d}f_{V}\left( \mathbf{x}\right) \left[ \log_2\prod\limits_{K\in 
\mathcal{C}}f_K\left( \mathbf{x}_K\right) -\log_2\prod\limits_{S\in \mathcal{S%
}}\left( f_S\left( \mathbf{x}_S\right) \right) ^{\nu _S-1}\right] d\mathbf{x}\vspace{3mm}\\
=-H\left( \mathbf{X}\right) -\int\limits_{R^d}f_{V}\left( \mathbf{x}\right)
\log_2\prod\limits_{K\in \mathcal{C}}f_K\left( \mathbf{x}_K\right) d\mathbf{x+%
}\int\limits_{R^d}f_{V}\left( \mathbf{x}\right) \log_2\prod\limits_{S\in 
\mathcal{S}}\left( f_S\left( \mathbf{x}_S\right) \right) ^{\nu _S-1}d\mathbf{x}.
\end{array}
\]
Since $\bigcup\limits_{K\in \mathcal{C}}K=V$ and each variable belongs
once more to the clusters than to the separators,  by adding and 
substracting 
\[
\int\limits_{R^d}f_{V}\left( \mathbf{x}\right) \log_2\prod\limits_{K\in 
\mathcal{C}}\prod\limits_{i\in K}f_i\left( x_i\right) d\mathbf{x}
\]
we obtain
\[
\begin{array}{l}
KL\left( \widehat{f}_{V_{\mathcal{CS}}}\left( \mathbf{x}\right) ,f_{V}\left(\mathbf{x}\right) \right) 
= -H\left( \mathbf{X}\right)-\int\limits_{R^d}f_{V}\left( \mathbf{x}\right)\log_2\dfrac{\prod\limits_{K\in \mathcal{C}}
f_K\left( \mathbf{x}_K\right) }{\prod\limits_{K\in \mathcal{C}}\prod\limits_{i\in K}f_i\left( x_i\right) }d\mathbf{x}\\
+\int\limits_{R^d}f_{V}\left( \mathbf{x}\right) 
\log_2\dfrac{\prod\limits_{S\in \mathcal{S}}\left[ f_S\left( \mathbf{x}_{_S}\right)\right] ^{\nu _S-1}}
{\prod\limits_{S\in \mathcal{S}}\left[\prod\limits_{i\in K}f_i\left( x_i\right) \right] ^{\nu _S-1}}d\mathbf{x}
-\int\limits_{R^d}f_{V}\left( \mathbf{x}\right)
\log_2\prod\limits_{i=1}^df_i\left( x_i\right) d\mathbf{x}\\
=-H\left( \mathbf{X}\right) -\int\limits_{R^d}f_{V}\left( \mathbf{x}\right)
\sum\limits_{K\in \mathcal{C}}\log_2\dfrac{f_K\left( \mathbf{x}_K\right) }{%
\prod\limits_{i\in K}f_i\left( x_i\right) }d\mathbf{x}\\
+\int\limits_{R^d}f_{V}\left( \mathbf{x}\right) \sum\limits_{S\in \mathcal{S}%
}\log_2\dfrac{\left[ f_S\left( \mathbf{x}_{_S}\right) \right] ^{\nu _S-1}}{%
\left[ \prod\limits_{i\in S}f_i\left( x_i\right) \right] ^{\nu _S-1}}d%
\mathbf{x}-\int\limits_{R^d}f_{V}\left( \mathbf{x}\right)
\sum\limits_{i=1}^d \log_2f_i\left( x_i\right) d\mathbf{x}
\end{array}
\]
Since $f_K\left( \mathbf{x}_K\right) ,$ $f_S\left( \mathbf{x}_{_S}\right) ,$ 
$f_i\left( x_i\right) $ are consistent marginals of $f_{V}\left( \mathbf{x}%
\right) $ we have the following relations:
\begin{equation}\label{eq:eq5}
\begin{array}{l}
\int\limits_{R^d}f_{V}\left( \mathbf{x}\right) \sum\limits_{K\in 
\mathcal{C}}\log_2\dfrac{f_K\left( \mathbf{x}_K\right) }{\prod\limits_{i\in
K}f_i\left( x_i\right) }d\mathbf{x}\\
=\sum\limits_{K\in \mathcal{C}%
}^{}\int\limits_{R^k}f_K\left( \mathbf{x}\right) \log_2\dfrac{f_K\left( 
\mathbf{x}_K\right) }{\prod\limits_{i\in K}f_i\left( x_i\right) }d\mathbf{x}%
_k\mathbf{=}\sum\limits_{K\in \mathcal{C}}^{}I\left( \mathbf{X}_K\right)
\end{array} 
\end{equation}
\begin{equation}\label{eq:eq6}
\begin{array}{l}
\int\limits_{R^d}f_{V}\left( \mathbf{x}\right) \sum\limits_{S\in 
\mathcal{S}}\log_2\dfrac{\left[ f_S\left( \mathbf{x}_{_S}\right) \right]
^{\nu _S-1}}{\prod\limits_{i\in S}f_i\left( x_i\right) }d\mathbf{x}\\
=\sum\limits_{S\in \mathcal{S}}^{}\left( \nu _S-1\right)
\int\limits_{R^{k-1}}f_S\left( \mathbf{x}\right) \log_2\dfrac{\left[ f_S\left( 
\mathbf{x}_{_S}\right) \right] }{\prod\limits_{i\in S}f_i\left( x_i\right) }d%
\mathbf{x}_S=\sum\limits_{S\in \mathcal{C}}^{}\left( \nu
_S-1\right) I\left( \mathbf{X}_S\right) 
\end{array}
\end{equation}
\begin{equation}\label{eq:eq7}
\begin{array}{l}
-\int\limits_{R^d}f_{V}\left( \mathbf{x}\right)
\sum\limits_{i=1}^d\log_2f_i\left( x_i\right)\\
=\sum\limits_{i=1}^d-\int\limits_{-\infty }^\infty f_i\left( \mathbf{x}%
_i\right) \log_2f_i\left( x_i\right) dx_i=\sum\limits_{i=1}^dH\left(
X_i\right) 
\end{array}
\end{equation}
where $I\left( \mathbf{X}_K\right) $, $I\left( \mathbf{X}_S\right) $ are the
information contents (see \cite{CoTo91}) of the $\mathbf{X}_K$ and $\mathbf{X}_S$
corresponding to the index set $K\in \mathcal{C}$ and $S\in \mathcal{S}$.

Taking into account relations (\ref{eq:eq5}), (\ref{eq:eq6}) and (\ref{eq:eq7}) we obtain:
\[
KL\left( \widehat{f}_{V_{\mathcal{CS}}}\left( \mathbf{x}\right) ,f_{V}\left(\mathbf{x}\right) \right) 
= \sum\limits_{i=1}^dH\left( X_i\right)-H\left( \mathbf{X}\right) -\left[ \sum\limits_{K\in \mathcal{C}_{ch}}I\left( \mathbf{X}_K\right) 
-\sum\limits_{S\in \mathcal{S}_{ch}}\left( \nu _S-1\right)I\left( \mathbf{X}_S\right) \right]  
\]
As we know that  
\[
\sum\limits_{i=1}^dH\left( X_i\right) -H\left( \mathbf{X}\right) =I\left( 
\mathbf{X}\right) 
\]
we obtained formula (\ref{eq:eq4}) and this proves the theorem.
\end{proof}

It is easy to see that the difference $I\left( \mathbf{X}\right) $ do not
depend on the structure of the junction tree . A consequence of Theorem \ref{theo:theo8} is the following remark.

\begin{remark}\label{rem:rem5}
The probability density $\widehat{f}_{V_{\mathcal{CS}}}$ of the form (\ref{eq:eq3}), which is
the best fitting cherry-wine to the real probability density $f_{V}$ 
over all possible truncated R-vines at level $k$ maximizes the following difference 
\[
\sum\limits_{K\in \mathcal{C}_{ch}}I\left( \mathbf{X}_K\right)
-\sum\limits_{S\in \mathcal{S}_{ch}}\left( \nu _S-1\right) I\left( \mathbf{X}%
_S\right).
\]
\end{remark}

Now we make some observation on the corresponding copula densities.

For two variables it was shown (see \cite{CaVi09} and \cite{MaSu08}) that:
\[
I\left( X,Y\right) =\int\limits_{\left[ 0;1\right] ^2}c\left( u,v\right)
\log_2 c\left( u,v\right) dudv
\]
which means that information content is equivalent with ''copula entropy''
concept introduced in \cite{MaSu08}.

Generalizing this for the variables involved in the sets $K$ and $S$ we have:
\[
\begin{array}{rcl}
I\left( \mathbf{X}_K\right)  & = &\int\limits_{\left[ 0;1\right] ^k}c\left( 
\mathbf{u}_{\mathbf{x}_K}\right) \log_2 c\left( \mathbf{u}_{\mathbf{x}_K}\right)
d\mathbf{u}_{\mathbf{x}_K}=-H\left( c_{\mathbf{x}_K}\right) \vspace{3mm} \\
I\left( \mathbf{X}_S\right)  &=&\int\limits_{\left[ 0;1\right]
^{k-1}}c\left( \mathbf{u}_{\mathbf{x}_S}\right) \log_2 c\left( \mathbf{u}_{%
\mathbf{x}_S}\right) d\mathbf{u}_{\mathbf{x}_S}=-H\left( c_{\mathbf{x}%
_S}\right) 
\end{array}
\]
Using the above assertions in Theorem \ref{theo:theo8} the Kullback-Leibler
divergence can be expressed by means of copula entropies:
\[
KL\left( \widehat{f}_{V_{\mathcal{CS}}}\left( \mathbf{x}_{V}\right) ,f_{V}\left( 
\mathbf{x}_{V}\right) \right) =H\left( c_{\mathbf{X}_{V}}\right) +\left[
\sum\limits_{K\in \mathcal{C}_{ch}}H\left( c_{\mathbf{X}_K}\right)
-\sum_{S\in \mathcal{S}_{ch}}\left( \nu _S-1\right) H\left( c_{\mathbf{x}%
_S}\right) \right] \text{ .}
\]

\begin{remark}
The cherry-wine copula density $\widehat{c}_{V_{\mathcal{CS}}}$ associated to
the best fitting cherry-wine probability density  $\widehat{f}_{V_{\mathcal{CS%
}}}$ minimizes the following difference over
all possible truncated R-vines at level $k$:
\end{remark}

\[
\sum\limits_{K\in \mathcal{C}_{ch}}H\left( c_{\mathbf{x}_K}\right)
-\sum\limits_{S\in \mathcal{S}_{ch}}\left( \nu _S-1\right) H\left( c_{%
\mathbf{x}_S}\right) \text{ .}
\]

\section{Conclusion}
\label{sec:6}

In this paper we gave an alternative definition of Regular-vines using the concept of $t$-cherry junction tree. We introduced the cherry-wine structure (a truncated R-vine assigned to a $t$-cherry probability distribution). We gave an algorithm for constructing a truncated R-vine at level $k$ starting from special $k$-th order $t$-cherry junction trees. The problem of inference was also discussed. We developed a method for obtaining a good factorization (which exploits conditional independences) starting from a sample data.
In the last section we discussed some theoretical properties of the best fitting truncated R-vine.
In future we are planning to extend our algorithm to the general case.



\begin{thebibliography}{}
%
%

\bibitem{Aasetal09}
K. Aas, C. Czado, A. Frigessi, and H. Bakken, Pair-copula constructions of multiple dependence, Insur. Math. Econ., 44, 182--198, (2009)

\bibitem{Bau11} A. Bauer, C. Czado and T. Klein, Pair-copula construction for non-Gaussian DAG models, submitted, (2011)

\bibitem{BedCoo01}
T. Bedford and R. Cooke, Probability density decomposition for conditionally dependent random variables modeled by vines, Ann. Math. Artif. Intell., 32, 245--268, (2001)

\bibitem{BedCoo02}
T. Bedford and R. Cooke, Vines -- a new graphical model for dependent random variables, Ann. Stat., 30(4), 1021--1068, (2002)

\bibitem{BreCzaAas10} E.C. Brechmann, C. Czado and K. Aas, Truncated regular vines in high dimensions
with applications to financial data, Submitted for publication, www-m4.ma.tum.de/Papers/Brechmann/vinetrunc.pdf (2010)

\bibitem{BuPre01}
J. Buksz\'{a}r and A. Pr\'{e}kopa, Probability Bounds with Cherry Trees, Mathematics of Operational Research 26, 174--192, (2001)

\bibitem{BuSza02}
J. Buksz\'{a}r and T. Sz\'{a}ntai, Probability Bounds given by hypercherry trees, Optimization Methods and Software, 17, 409--422, (2002)

\bibitem{CaVi09}
R.S. Calsaverini and R. Vicente, 
An information theoretic approach to statistical dependence: Copula information, 
arXiv:0911.4207v1, (2009)

\bibitem{Chow68}
C.K. Chow and C.N. Liu, Approximating Discrete Probability Distribution with Dependence Tree,
IEEE Transactions on Informational Theory, 14, 462-467, (1968)

\bibitem{CoTo91}
T.M. Cover and J.A. Thomas, Elements of Information Theory, Wiley Interscience, New York, (1991)

\bibitem{CowDawLauSpi03}
R.G. Cowell, P.A. Dawid, S.L. Lauritzen and D.J Spiegelhalter,
Probabilistic Networks and Expert Systems (Information Science and Statistics), Springer,
Heidelberg, (2003)

\bibitem{Cza10} C. Czado, Pair-copula constructions of multivariate copulas, In: P. Jaworski, F. Durante, W. Härdle and T. Rychlik (Eds.), Copula Theory and Its Applications, Berlin, Springer, (2010)

\bibitem{Dis10} J.F. Dissmann, Statistical Inference for Regular Vines and Application,
Thesis, Technical University of Munich, Center of Mathematics, (2010)

\bibitem{GenRiv93}
C. Genest and L.P. Rivest, Statistical inference procedures for bivariate Archimedean copulas, J. Am. Stat. Assoc., 88(423), 1034--1043, (1993)

\bibitem{HafAasFri09}
I.H. Haff, K. Aas and A. Frigessi, On the simplified pair-copula construction -- simply useful or too simplistic? Technical Report, Norvegian Computing Center, Oslo (2009)

\bibitem{Ham71}
J.M. Hammersley, Markov fields on finite graphs and lattices, not published, (1971)

\bibitem{HanKurCoo06}
A. Hanea, D. Kurowicka and R. Cooke, Hybrid method for quantifying and analyzing Bayesian belief networks, Qual. Reliab. Eng., 22, 708--729, (2006)

\bibitem{Joe97}
H. Joe, Multivariate Models and Dependence Concepts, Chapman \& Hall, London, (1997)

\bibitem{KarSre01}
D. Karger and N. Srebro, Learning networks: Maximum likelihood bounded
tree-width graphs, SODA-01

\bibitem{KoSza10}  
E. Kov\'{a}cs and T. Sz\'{a}ntai, 
On the approximation of discrete multivariate probability distribution using the
new concept of $t$-cherry junction tree, Lecture Notes in Economics
and Mathematical Systems, 633, Proceedings of the
IFIP/IIASA/GAMM Workshop on Coping with Uncertainty, Robust Solutions, 2008,
IIASA, Laxenburg, 39--56, (2010)

\bibitem{KoSza10a} E. Kov\'{a}cs and T. Sz\'{a}ntai, Multivariate copula expressed by lower dimensional copulas, http://arxiv.org/abs/1009.2898, (2010)

\bibitem{Kul59}
S. Kullback, Information Theory and Statistics, Wiley and Sons, New York, (1959)

\bibitem{Kur10}
D. Kurowicka, Optimal truncation of vines, in: D. Kurowicka and H. Joe (eds) Dependence-Modeling -- Handbook on Vine Copulas, Word Scientific Publishing, Singapore, (2010)

\bibitem{KurCoo02} D. Kurowicka and R. Cooke, The vine copula method for representing high dimensional dependent distributions: Application to continuous belief nets, Proceedings of the 2002 Winter Simulation Conference, 270--278, (2002)

\bibitem{KurCoo06} D. Kurowicka and R. M. Cooke, Uncertainty Analysis with High Dimensional Dependence Modelling, Chichester, John Wiley, (2006)

\bibitem{LauSpi88}
S.L. Lauritzen and D.J. Spiegelhalter, Local Computations with Probabilites
on Graphical Structures and their Application to Expert Systems, J.R.
Statist. Soc. B, 50, 157--227, (1988)

\bibitem{MaSu08}
J. Ma and Z. Sun, Mutual information is copula entropy, 
arXiv: 0808.0845v1, (2008)

\bibitem{MinCza10}
A. Min and C. Czado, Bayesian inference for multivariate copulas using pair-copula constructions, J. Financ. Econom, to appear, preprint available under: http://www-m4.ma.tum.de/Papers/index.html, (2010)

\bibitem{MinCza09}
A. Min and C. Czado, Bayesian model selection for multivariate copulas using pair-copula constructions, Preprint, (2009)

\bibitem{MorCooKur10}
O. Morales-Napoles, R. Cooke and D. Kurowicka, About the number of vines and regular vines on $n$ nodes, Discrete Appl. Math., Submitted  (2010)

\bibitem{Ne99} 
R. Nelsen, An Introduction to Copulas, Springer, New York, (1999)

\bibitem{Pea88}
J. Pearl, Probabilistic reasoning in intelligent systems, CA:
Morgan Kauffman, Palo Alto, (1988)

\bibitem{SzaKov08}  
Sz\'{a}ntai, T. and E. Kov\'{a}cs, 
Hypergraphs as a mean of discovering the dependence structure of a discrete multivariate probability
distribution, \textit{Proc. Conference APplied mathematical programming and 
MODelling (APMOD), 2008}, Bratislava, 27-31 May 2008,
Annals of Operations Research, to appear.

\bibitem{SzaKov10}
T. Sz\'{a}ntai and E. Kov\'{a}cs, 
Application of $t$-cherry junction trees in pattern recognition, 
Broad Research in Artificial Intelligence and Neuroscience (BRAIN), Special Issue on Complexity in Sciences and Artificial Intelligence,
Eds. B. Iantovics, D. Radoiu, M. Marusteri and M. Dehmer, 40--45, (2010)

\bibitem{SzaKov10a} T. Sz\'{a}ntai and E. Kov\'{a}cs, Discovering a junction tree behind a Markov network by a greedy algorithm, http://arxiv.org/abs/1104.2762, (2010)

\bibitem{TarYan84}
R.E.Tarjan and M. Yanakakis, Simple linear-time algorithms to test
chordality of graphs, test acyclity of hypergraphs and selectively reduce
acyclic hypergraphs, SIAM J. Comp.,13, 566--579, (1984)

\bibitem{Yan81}
M.Yanakakis, Computing the minimum Fill in is NP-complete, SIAM Journal
Alg. Disc. Math., 2, 77--79, (1981)
\end{thebibliography}


\end{document}